\documentclass[a4paper,10pt]{article}
\usepackage{hyperref}
\usepackage[T1]{fontenc}
\usepackage[utf8]{inputenc}

\usepackage{todonotes}

\usepackage[english]{babel}
\usepackage{mathtools}
\usepackage{amsfonts}
\usepackage{color}
\usepackage{amsmath}
\usepackage{amsthm}
\usepackage{amssymb}
\usepackage{mathrsfs}
\usepackage{amstext}

\usepackage[normalem]{ulem}

\usepackage{graphicx}
\usepackage{tikz}
\usepackage{subfig}
\usepackage{pgfplots}
\usepgfplotslibrary{fillbetween}
\pgfplotsset{compat=1.18}
\usepackage{subcaption}

\usepackage{authblk}
\usepackage{stmaryrd}
\usepackage{makeidx}
\makeindex

\numberwithin{equation}{section}

\theoremstyle{definition}
\newtheorem{thm}{Theorem}
\newtheorem{rem}[thm]{Remark}
\newtheorem{defi}[thm]{Definition}

\newtheorem{cor}[thm]{Corollary}

\newtheorem*{thm*}{Theorem}
\newtheorem*{rem*}{Remark}
\newtheorem*{folg*}{Folgerung}
\newtheorem*{examples*}{Beispiele}
\newtheorem*{ex*}{Beispiel}
\newtheorem*{lem*}{Lemma}
\newtheorem*{prop*}{Proposition}
\newtheorem*{defi*}{Definition}
\newtheorem*{exercise*}{Übung}
\newtheorem*{conj*}{Conjecture}
\newtheorem*{q*}{Question}

\usepackage{cleveref}

\newcommand{\be}{\mathbf{e}}

\newcommand{\bm}{\mathbf{m}}

\newcommand{\bs}{\mathbf{s}}
\newcommand{\bt}{\mathbf{t}}

\newcommand{\bx}{\mathbf{x}}

\newcommand{\boldy}{\mathbf{y}}

\newcommand{\bbN}{\mathbb{N}}

\newcommand{\bbR}{\mathbb{R}}

\newcommand{\bbT}{\mathbb{T}}

\newcommand{\bbZ}{\mathbb{Z}}

\newcommand{\cT}{\mathcal{T}}

\newcommand{\fS}{\mathfrak{S}}

\newcommand{\iu}{\mathrm{i}}

\newcommand{\sm}{\setminus}
\newcommand{\sbse}{\subseteq}

\newcommand{\sqf}{\sqrt{5}}

\newcommand{\Lp}{L_2^{\text{per}}}

\newcommand{\eps}{\varepsilon}

\newcommand{\dx}{\text{d}}

\author{Nicolas Nagel\footnote{E-mail: \texttt{nicolas.nagel@mathematik.tu-chemnitz.de} \\ ORCID iD: 0009-0004-3362-3543}}

\title{On the Global Optimality of Fibonacci Lattices in the Torus}
\affil{Chemnitz University of Technology, Department of Mathematics}
\date{}

\pgfplotsset{yticklabel style={text width=3em,align=right}}

\allowdisplaybreaks

\begin{document}
	
	\maketitle
	
	
\begin{abstract}
	We use linear programming bounds to analyze point sets in the torus with respect to their optimality for problems in discrepancy theory and quasi-Monte Carlo methods. These concepts will be unified by introducing tensor product energies.
	
	We show that the canonical $3$-point lattice in any dimension is globally optimal among all $3$-point sets in the torus with respect to a large class of such energies. This is a new instance of universal optimality, a special phenomenon that is only known for a small class of highly structured point sets.
	
	In the case of $d=2$ dimensions it is conjectured that so-called Fibonacci lattices should also be optimal with respect to a large class of potentials. To this end we show that the $5$-point Fibonacci lattice is globally optimal for a continuously parametrized class of potentials relevant to the analysis fo the quasi-Monte Carlo method.
\end{abstract}

\section{Introduction}

\subsection{Preface on energy}

The notion of an energy of a point set is by now a classical one in physics, geometry and coding theory. Simply put, this asks to find a configuration of a given number of points in some space such that the total energy of the system, that is the sum of all the potential interactions between the points, is minimized. The case of the sphere (or more generally of two-point homogeneous spaces \cite{Wan52}, such as the projective plane) has garnered particular attention over the last couple of decades \cite{ADGMS23, And93, BD19, BDM16, BGMPV21, BGMPV22, BM19, BGM24, BG17, BDM99, BHS12, CGGKO20, CK07, DGS91, Gla23, GP20, Skr22, Yud93}. There, a straight forward notion of energy of a finite point set $X \sbse S^d \sbse \bbR^{d+1}$ for a potential function $c: (0, 2] \rightarrow \bbR$ is given by
$$
\sum_{\substack{\bx, \boldy \in X \\ \bx \neq \boldy}} c(\|\bx-\boldy\|),
$$
where $\|\cdot\|$ denotes the Euclidean norm in $\bbR^{d+1}$. A classical example would be Thomson's problem \cite{And93, Sch13, Tho04, Yud93} on $S^2$ with the (scaled) Coulomb potential $c(r) = r^{-1}$. In this particular case, the optimal configurations for $N=1, 2, 3, 4, 5, 6$ and $12$ points are known. Much more general, for certain dimensions $d$ and certain numbers of points $N$ there are point sets that are the minimizers of a large class of potentials simultaneously, a phenomenon known as \emph{universal optimality} \cite{CK07}.

In this paper we will consider the case of the torus, denoted by $[0, 1)^d$ (in the literature also found as $\bbT^d$) with opposite faces identified. Consequently, functions on $[0, 1)^d$ will be considered as periodic functions, extended to $\bbR^d$ in the obvious way if desired. So far, energies on the torus have mostly been studied where the interaction between particles depends on the Euclidean distance between them. In this sense, every point on the torus $[0, 1)^d$ is associated to a corresponding shifted (and possibly skewed) lattice in $\bbR^d$ and all possible interactions from a given point to all points of the corresponding lattice are considered, see \cite{CKMRV22, Cou06, CS12, HT23, SS06, Sch10, Ten23} for some references. In this context, tools from complex analysis and modular forms are frequently used.

To motivate the potentials that we will consider we start by discussing some topics concerning uniformly distributed point sets in $[0, 1)^d$.

\subsection{Quasi-Monte Carlo integration, discrepancy and diaphony}

Let $p > 0$ be a parameter. Consider the space $H^1_p$ of $1$-periodic, absolutely continuous functions $f: [0, 1) \rightarrow \bbR$ with
$$
\|f\|_{H^1_p}^2 \coloneqq \left(\int_0^1 f(x) \, \dx x\right)^2 + \frac 1p \int_0^1 f'(x)^2 \,\dx x 
$$
($f'$ denoting the first weak derivative of $f$). This space turns out to be a reproducing kernel Hilbert space with kernel function \cite{BT11, HO16, Wah95}
$$
K_p^1(x,y) = 1 + \frac p2 \left(\frac16 - |x-y| + |x-y|^2\right),
$$
that is function evaluation at any point $y \in [0, 1)$ can be realized via a scalar product with $K(\cdot, y)$ in $H^1_p$:
$$
f(y) = \left(\int_0^1 f(x) \,\dx x\right) \left(\int_0^1 K(x, y) \,\dx x\right) + \frac 1p \int_0^1 f'(x) \frac\partial{\partial x} K(x, y) \,\dx x
$$
for all $f \in H_p^1$ (in particular, function evaluation is continuous). Taking $d$-fold tensor products $H_p^d = H_p^1 \otimes ... \otimes H_p^1$ gives a reproducing kernel Hilbert space of functions $f: [0, 1)^d \rightarrow \bbR$ (a Sobolev space of periodic functions of dominating mixed smoothness) having the kernel
$$
K_p^d(\bx, \boldy) = \prod_{i=1}^d K_p^1(x_i, y_i)
$$
for $\bx, \boldy \in [0, 1)^d$. With this at hand we can consider the quasi-Monte Carlo integration algorithm \cite{DP10, Ham64, KN74, Nie92} for $f \in H_p^d$, that is we approximate
$$
\int_{[0, 1)^d} f(\bx) \,\dx\bx \approx \frac1N \sum_{\bx \in X} f(\bx)
$$
where $X \sbse [0, 1)^d$ is a finite set of sample nodes with $\#X = N$. The quality of this approximation depends on the choice of evaluation nodes $X$ and can be quantified via the worst case error \cite{DP10, NW10}
\begin{align} \label{eq:wce_kpd}
	\sup\limits_{\|f\|_{H_p^d} \leq 1} \left|\int_{[0, 1)^d} f(\bx) \,\dx\bx - \frac1N \sum_{\bx \in X} f(\bx)\right|^2 = -1 + \frac1{N^2} \sum_{\bx, \boldy \in X} K_p^d(\bx, \boldy).
\end{align}
The expression on the right hand side is related to the mainly geometric notion of periodic $L_2$-discrepancy on the torus \cite{DHP20, HKP21, Lev95}
$$
\Lp(X)^2 \coloneqq -\frac{N^2}{3^d} + \sum_{\bx, \boldy \in X} \prod_{i=1}^d \left(\frac12 - |x_i-y_i| + |x_i-y_i|^2\right)
$$
(unnormalized, as given in \cite{HKP21}), in fact for $p=6$ it only differs from the worst case integration error by a factor of $N^2/3^d$. It is also similar to the notion of diaphony \cite{Lev95, Zin76}
$$
F_N(X)^2 \coloneqq \sum_{\bm \in \bbZ^d - \mathbf{0}} \left(\prod_{i=1}^d \max\{1, |m_i|\}\right)^{-2} \left|\frac1N \sum_{\bx \in X} \exp(2\pi\iu \bm^\top \bx)\right|^2
$$
see \cite{HKP21, Pil23} for details. For small $N$ or $d$ optimal point configurations minimizing the discrepancy have been determined in \cite{PVC05, Whi77} (note that they use a slightly different notion of discrepancy) and \cite{HO16}.

\subsection{Tensor product energies}

Another interpretation, more akin to the notion of energy, is to observe that minimizing the worst case integration error \eqref{eq:wce_kpd} is equivalent to minimizing
\begin{align} \label{eq:energy_kpd}
	\sum_{\substack{\bx, \boldy \in X \\ \bx \neq \boldy}} K_p^d(\bx, \boldy) = \sum_{\substack{\bx, \boldy \in X \\ \bx \neq \boldy}} \prod_{i=1}^d \left(1+\frac p2\left(\frac16 - |x_i-y_i| + |x_i-y_i|^2\right)\right).
\end{align}
We can think of every point as a particle whose interaction is given by some potential function and the potential between any two points $\bx$ and $\boldy$ only depends on
\begin{align} \label{torus_geometry}
	(\min\{|x_1-y_1|, 1-|x_1-y_1|\}, \dots, \min\{|x_d-y_d|, 1-|x_d-y_d|\})
\end{align}
(since $t^2-t = (1-t)^2 - (1-t)$ for $0 \leq t \leq 1$). The terms in the components are respectively the geodesic distances between two points in the torus $[0, 1)$ (similar to (8) in \cite{HO16}). Minimizing the worst case integration error is thus equivalent to minimizing this sort of energy given by the kernel function $K_p^d$.

In the case of $d = 2$ and $p \in \{1, 6\}$ the minimizers of \eqref{eq:energy_kpd} were determined in \cite{HO16} for $N \leq 16$ points. Notably, it was observed that if $N = F_\ell$ is a Fibonacci number (initializing $F_0=0, F_1=1$) in this range, namely $1, 2, 3, 5, 8, 13$, the algorithm found that the \emph{Fibonacci lattices} given by
\begin{align} \label{eq:fib_lat}
	\Phi_\ell \coloneqq \left\{\left(\frac{m}{F_\ell}, \left\{\frac{m F_{\ell-1}}{F_\ell}\right\}\right): m=0, 1, \dots, F_\ell-1\right\}
\end{align}
are the global optimizers for the worst case integration error for their respective numbers of points. Prior to that, it was already known that the Fibonacci lattices are at least asymptotically of optimal order, matching (up to a multiplicative constant) the lower bound given by (for $d = 2$)
$$
-1 + \frac1{N^2} \sum_{\bx, \boldy \in X} K_p^2(\bx, \boldy) \gtrsim \frac{\log N}{N^2},
$$
see \cite{BTY12-1, BTY12-2, Bor17, DTU18, HKP21, HMOU16, Rot54, Tem92}. 
In this paper we investigate the question whether Fibonacci lattices are also optimal with respect to a more general notion of energy. The following definition summarizes our setup.

\begin{defi} \label{def:energy}
	Let $c: [0, 1)^d \rightarrow \bbR$ be a $1$-periodic function in every component. We say that $c$ is a \textbf{potential} if $c(t_1, \dots, t_d) = c(t_{\tau(1)}, \dots, t_{\tau(d)})$ for any permutation $\tau$ of $[d]$ and $c(t_1, t_2, \dots, t_d) = c(1-t_1, t_2, \dots, t_d)$ for all $\bt \in [0, 1)^d$. For a finite point set $X \sbse [0, 1)^d$ the \textbf{tensor product energy} with respect to the potential $c$ is given by
	$$
	E_c(X) \coloneqq \sum_{\substack{\bx, \boldy \in X \\ \bx \neq \boldy}} c(\bx - \boldy).
	$$
\end{defi}

The conditions on the potential guarantee that the potential between two point sets only depends on \eqref{torus_geometry}. A special example of a tensor product energy is given by \eqref{eq:energy_kpd} with (parametrized) potential
\begin{align} \label{eq:potential_for_kpd}
	c_p^d(\bt) = \prod_{i=1}^d \left(1 + \frac p2\left(\frac16 - t_i + t_i^2\right)\right).
\end{align}
This potential tensorizes, hence the name ``tensor product potential'', although for the sake of generality we have defined it more broadly.

\subsection{Main results}

The main contribution of this paper is to show that certain point sets are globally optimal with respect to a wide range of potentials. We start with the following result concerning the $5$-point Fibonacci lattice.

\begin{thm} \label{thm:fib_lat_5}
	Let $d=2$ and let the potential $c = c_p^2$ be given by \eqref{eq:potential_for_kpd}. Then the $5$-point Fibonacci lattice
	$$
	\Phi_5 = \left\{(0, 0), \left(\frac15, \frac35\right), \left(\frac25, \frac15\right), \left(\frac35, \frac45\right), \left(\frac45, \frac25\right)\right\}
	$$
	minimizes $E_{c_p^2}$ among all $5$-point sets in $[0, 1)^2$ for the range $0 \leq p \leq 9$.
\end{thm}

This generalizes the results from \cite{HO16} for $N=5$. The method there is computationally heavy and can only check that $\Phi_5$ is optimal for any individual $p$. Here however we will get the optimality of an entire range by checking only a few, elementary inequalities.

For the next theorem we need to introduce a certain transformation.

\begin{defi} \label{defi:arccos_trafo}
	For a potential $c: [0, 1)^d \rightarrow \bbR$ define the \textbf{arccos-transform}
	$$
	\gamma: [-1, 1]^d \rightarrow \bbR, \gamma(s_1, \dots, s_d) \coloneqq c\left(\frac{\arccos s_1}{2\pi}, \dots, \frac{\arccos s_d}{2\pi}\right).
	$$
\end{defi}

Under this transformation we have 
$$
E_c(X) = \sum_{\substack{\bx, \boldy \in X \\ \bx \neq \boldy}} \gamma\left(\cos(2\pi(x_1-y_1)), \dots, \cos(2\pi(x_d-y_d))\right).
$$

\begin{thm} \label{thm:3_pt_lat}
	Let the dimension $d \in \bbN$ be arbitrary and $c: [0, 1)^d \rightarrow \bbR$ be a potential with the corresponding arccos-transform $\gamma: [-1, 1]^d \rightarrow \bbR$. Assume that $\gamma$ is convex and increasing in the sense that if $\bs^1 \leq \bs^2$ componentwise then $\gamma(\bs^1) \leq \gamma(\bs^2)$. Then
	$$
	\cT^d \coloneqq \left\{\frac13 (k, \dots, k): k = 0, 1, 2\right\}
	$$
	minimizes $E_c$ among all $3$-point sets in $[0, 1)^d$.
\end{thm}

This theorem gives the \emph{universal optimality property} for the $3$-point lattice $\cT^d$ in any dimension, meaning that it is the global minimizer among all $3$-point sets for a large, natural class of potentials. Observe that any $3$-point rational lattice (see Definition \ref{def:rat_lat} below) in $[0, 1)^d$ can be transformed into $\cT^d$ under a torus symmetry (coordinate shifts, reflections and permutations) so that we might speak of \emph{the} rational $3$-point lattice. In particular, in $d=2$ dimensions, $\cT^2$ is equivalent to the $3$-point Fibonacci lattice $\Phi_4$.

\begin{rem}
	Note that the case for $N=2$ is easily dealt with, since then the point set $\{(0, \dots, 0), (1/2, \dots, 1/2)\}$ is the minimizer for a potential $c$ if and only if $c$ has a global minimum in $\bt = (1/2, \dots, 1/2)$.
\end{rem}

\subsection{Notation}

The set of integers will be denoted by $\bbZ$, the set of nonnegative numbers by $\bbN_0 = \{0, 1, \dots\}$. We will also use $[n] = \{1, \dots, n\}$. The cardinality of a set $X$ will by written using $\#X$. The set difference with a singleton will be written as $X - x \coloneqq X \sm \{x\}$.

Vectors will be written in bold type (for example $\bm, \bs, \bx$), their components by usual type (for example $m_i, s_i, x_i$). Similarly, the all-zeros and all-ones vectors will be denoted by $\mathbf{0}$ and $\mathbf{1}$ respectively.

The coefficients of a decomposition of a function $f$ into an orthogonal basis (notably into a Fourier series or via Chebyshev polynomials) will be denoted by $\hat f$. It will be clear from context which basis functions are used.

The fractional part of $x \in \bbR$ will be denoted by $\{x\} \coloneqq x - \lfloor x \rfloor$ (where context will make it clear that this is not a singleton set).

\section{The linear programming bound}

Let us state here the main tool for proving lower bounds. This is an adaptation of the method on the sphere, see \cite{CK07, DGS91, Yud93}, to the torus.

\begin{thm} \label{thm:continuous_bound}
	Let $c: [0,1 )^d \rightarrow \bbR$ be a potential. Let $b: [0, 1)^d \rightarrow \bbR$ be such that $c(\bt) \geq b(\bt)$ for all $\bt \in [0, 1)^d$ and for its Fourier expansion we have
	$$
	b(\bt) = \sum_{\bm \in \bbZ^d} \hat{b}(\bm) \exp(2\pi\iu \bm^\top \bt)
	$$
	(converging pointwise) with $\hat{b}(\bm) \geq 0$ for all $\bm \in \bbZ^d - \mathbf{0}$. Then
	$$
	E_c(X) \geq N^2 \hat{b}(\mathbf{0}) - N b(\mathbf{0})
	$$
	for all $X \subseteq [0, 1)^d$ with $\#X = N$.
\end{thm}

A function on the torus $[0, 1)^d$ all of whose Fourier coefficients are nonnegative is also called \emph{positive definite}. The pointwise convergence of $b$ will turn out to not be a problem for us since the functions that we will construct further below only us a finite set of indices, that is we only consider trigonometric polynomials.

\begin{proof} [Proof of Theorem \ref{thm:continuous_bound}]
	Let $b$ be as in the assumptions, then
	\begin{align} \label{ineq:lp}
		\begin{split}
			E_c(X) & = \sum_{\substack{\bx, \boldy \in X \\ \bx \neq \boldy}} c(\bx - \boldy) \geq \sum_{\substack{\bx, \boldy \in X \\ \bx \neq \boldy}} b(\bx - \boldy) = - N b(\mathbf{0}) + \sum_{\bx, \boldy \in X} b(\bx - \boldy) \\
			& = - N b(\mathbf{0}) + \sum_{\bm \in \bbZ^d} \hat{b}(\bm) \sum_{\bx, \boldy \in X} \exp(2\pi\iu \bm^\top (\bx - \boldy)) \\
			& = - N b(\mathbf{0}) + \sum_{\bm \in \bbZ^d} \hat{b}(\bm) \left|\sum_{\bx\in X} \exp(2\pi\iu \bm^\top \bx)\right|^2 \\
			& \geq N^2 \hat{b}(\mathbf{0}) - N b(\mathbf{0}).
		\end{split}
	\end{align}
\end{proof}

For the analysis we will reformulate this statement in terms of $\gamma$. For this we will need the \emph{Chebyshev polynomials}
$$
T_m(s) \coloneqq \cos(m \arccos s).
$$

\begin{cor} \label{cor:arccos_transformed}
	Let $c: [0, 1)^d \rightarrow \bbR$ be a potential with arccos-transform $\gamma: [-1, 1]^d \rightarrow \bbR$. Let $\beta: [-1, 1]^d \rightarrow \bbR$ fulfill $\gamma(\bs) \geq \beta(\bs)$ for all $\bs \in [-1, 1]^d$ and
	$$
	\beta(\bs) = \sum_{\bm \in \bbN_0^d} \hat\beta(\bm) \prod_{i=1}^d T_{m_i}(s_i)
	$$
	(converging pointwise) with $\hat\beta(\bm) \geq 0$ for all $\bm \in \bbN_0^d - \mathbf{0}$. Then
	$$
	E_c(X) \geq N^2 \hat\beta(\mathbf{0}) - N \beta(\mathbf{1})
	$$
	for all $X \sbse [0, 1)^d, \#X = N$.
\end{cor}

\begin{proof}
	Let $\beta$ be as in the assumption and consider the function $b: [0, 1)^d \rightarrow \bbR$ given by
	\begin{align} \label{eq:coeff_b_beta}
		b(\bt) = \sum_{\bm \in \bbZ^d} 2^{-\#\{i \in [d]: m_i \neq 0\}} \hat\beta(|m_1|, \dots, |m_d|) \exp(2\pi\iu \bm^\top \bt).
	\end{align}
	We will show that $b$ is a valid function fulfilling the requirements of Theorem \ref{thm:continuous_bound}. Indeed, by definition it is clear that $b$ only has nonnegative Fourier coefficients and it remains to check that $c \geq b$. For this note that
	\begin{align*}
		b(\bt) & = \sum_{\bm \in \bbN_0^d}  \hat\beta(\bm) \sum_{\substack{\underline\eps \in \{-1, 1\}^d \\ \eps_i \in \fS(m_i)}} \prod_{i=1}^d \frac{\exp(2\pi\iu \eps_i m_i t_i)}{\#\fS(m_i)} \\
		& = \sum_{\bm \in \bbN_0^d} \hat\beta(\bm) \prod_{i=1}^d \left(\frac1{\#\fS(m_i)}\sum_{\substack{\eps \in \fS(m_i)}} \exp(2\pi\iu \eps m_i t_i)\right) \\
		& = \sum_{\bm \in \bbN_0^d} \hat\beta(\bm) \prod_{i=1}^d \cos(2\pi m_i t_i),
	\end{align*}
	where $\underline\eps = (\eps_1, \dots, \eps_d)$ and
	$$
	\fS(m) \coloneqq \begin{cases}
		\{1\} & , m = 0 \\
		\{-1, 1\} & , m \neq 0.
	\end{cases}
	$$
	In particular we get that
	$$
	b\left(\frac{\arccos s_1}{2\pi}, \dots, \frac{\arccos s_d}{2\pi}\right) = \sum_{\bm \in \bbN_0^d} \hat\beta(\bm) \prod_{i=1}^d T_{m_i}(s_i) = \beta(\bs),
	$$
	that is $\beta$ is the arccos-transform of $b$. By $\gamma \geq \beta$ we thus have $c(\bt) \geq b(\bt)$ for $\bt \in [0, 1/2]^d$ (the range of the function $s \mapsto
	\arccos(s)/2\pi$). However, by \eqref{torus_geometry}, which by construction is also fulfilled by $b$, the functions $c$ and $b$ are completely determined by their values on $[0, 1/2]^d$, so that we even have $c(\bt) \geq b(\bt)$ for all $\bt \in [0, 1)^d$. We may thus apply Theorem \ref{thm:continuous_bound}, where it remains to note that $\hat b(\mathbf{0}) = \hat\beta(\mathbf{0})$ and $b(\mathbf{0}) = \beta(\mathbf{1})$.
\end{proof}

Let us finish this section by inspecting when we get equality in the LP-bound. Let us start by the following special class of point sets.

\begin{defi} \label{def:rat_lat}
	A \textbf{rational lattice} is a finite set $\Lambda \sbse [0, 1)^d, \#\Lambda = N$ of the form
	$$
	\left\{\left(\frac mN, \left\{\frac{m h_2}N\right\}, \dots, \left\{\frac{m h_d}N\right\}\right): m=0, 1, \dots, N-1\right\},
	$$
	where $h_2, \dots, h_d \in \{1, \dots, N-1\}$ with $\gcd(h_2, N) = \dots = \gcd(h_d, N) = 1$. The vector $\frac1N(1, h_2, \dots, h_d)$ is called the \textbf{generating vector} of $\Lambda$.
\end{defi}

\begin{thm} \label{thm:magic_function}
	Let $\Lambda \sbse [0, 1)^d$ be a rational lattice of size $\#\Lambda = N < \infty$ and generating vector $\frac1N(1, h_2, \dots, h_d)$. Let $c: [0, 1)^d \rightarrow \bbR$ be a potential with arccos-transform $\gamma: [-1, 1]^d \rightarrow \bbR$. Let $\beta: [-1, 1]^d \rightarrow \bbR$ be such that:
	\begin{itemize}
		\item [(i)] $\gamma(\bs) \geq \beta(\bs)$ for all $\bs \in [-1, 1]^d$,
		
		\item [(ii)] $\beta(\bs) = \sum_{\bm \in \bbN_0^d} \hat\beta(\bm) \prod_{i=1}^d T_{m_i}(s_i)$ (converging pointwise) with $\hat\beta(\bm) \geq 0$ for all $\bm \in \bbN_0^d-\mathbf{0}$,
		
		\item [(iii)] $\beta(\cos(2\pi s_1), \dots, \cos(2\pi s_d)) = \gamma(\cos(2\pi s_1), \dots, \cos(2\pi s_d))$ for all $\bs \in \Lambda - \mathbf{0}$,
		
		\item [(iv)] $\hat\beta(\bm) = 0$ for all $\bm \in \bbN_0^d - \mathbf{0}$ for which there are signs $(\eps_1, \dots, \eps_d) \in \{-1, 1\}^d$ such that $\eps_1 m_1 + \eps_2 m_2 h_2 + \dots + \eps_d m_d h_d \equiv 0 \mod N$.
	\end{itemize}
	Then $E_c(X) \geq E_c(\Lambda)$ for all $X \sbse [0, 1)^d, \#X = N$.
\end{thm}

\begin{proof}
	Let $X \sbse [0, 1)^d, \#X = N$ be arbitrary. By Corollary \ref{cor:arccos_transformed} conditions (i) and (ii) give
	$$
	E_c(X) \geq N^2 \hat\beta(\mathbf{0}) - N \beta(\mathbf{1}).
	$$
	Condition (iii) and the lattice structure of $\Lambda$ imply that
	\begin{align*}
		E_c(\Lambda) = & \sum_{\substack{\bx, \boldy \in \Lambda \\ \bx \neq \boldy}} c(\bx - \boldy) \\
		= & \sum_{\substack{\bx, \boldy \in \Lambda \\ \bx \neq \boldy}} \gamma(\cos(2\pi(x_1-y_1)), \dots, \cos(2\pi(x_d-y_d))) \\
		= & \sum_{\substack{\bx, \boldy \in \Lambda \\ \bx \neq \boldy}} \beta(\cos(2\pi(x_1-y_1)), \dots, \cos(2\pi(x_d-y_d))) \\
		= & \sum_{\substack{\bx, \boldy \in \Lambda \\ \bx \neq \boldy}} b(\bx - \boldy),
	\end{align*}
	so that we have equality in the first inequality of \eqref{ineq:lp}. As for the second inequality therein, note that one would need to have
	$$
	\hat b(\bm) \left|\sum_{\bx \in \Lambda} \exp(2\pi\iu \bm^\top \bx)\right|^2 = 0
	$$
	for all $\bm \in \bbZ^d - \mathbf{0}$. Thus at least one of $\hat b(\bm)$ and
	$$
	\sum_{\bx \in \Lambda} \exp(2\pi\iu \bm^\top \bx) = \sum_{k=0}^{N-1} \exp\left(\frac{2\pi\iu k}N (m_1 + m_2 h_2 + \dots + m_d h_d)\right)
	$$
	must be zero. In particular, if the sum is nonzero, which is the case if $m_1 + m_2 h_2 + \dots + m_d h_d \equiv 0 \mod N$, then the corresponding coefficients $\hat b(\bm)$ must be. Since the coefficients of $b$ and $\beta$ are related by \eqref{eq:coeff_b_beta}, condition (iv) implies equality in the second inequality of \eqref{ineq:lp}.
\end{proof}

\begin{defi} \label{def:forbidden_degree}
	A vector $\bm \in \bbN_0^d - \mathbf{0}$ as in condition (iv) above will be called a \textbf{forbidden degree} (with respect to the lattice $\Lambda$). A function $\beta$ fulfilling the conditions (i-iv) above is called a \textbf{magic function} for the lattice $\Lambda$.
\end{defi}

\section{Universal optimality of the $3$-point lattices}

Let us now show the universal optimality of the $3$-point lattice $\cT^d$. 

\begin{proof} [Proof of Theorem \ref{thm:3_pt_lat}]
	The result follows directly from Theorem \ref{thm:magic_function} by constructing a suitable magic function $\beta$. Take a supporting hyperplane 
	$$
	\beta(\bs) = \beta_0 + \beta_1 (s_1 + \dots + s_d)
	$$
	of $\gamma$ in the point
	$$
	\left(\cos\left(\frac{2\pi}3\right), \dots, \cos\left(\frac{2\pi}3\right)\right) = \left(-\frac12, \dots, -\frac12\right).
	$$
	By convexity we have $\gamma(\bs) \geq \beta(\bs)$ with equality for $\bs = -\frac12 \mathbf{1}$ and by monotonicity we have $\beta_1 \geq 0$. Also note that the degrees
	$$
	\be^i = (0, \dots, 0, 1, 0, \dots, 0) \in \bbN_0^d
	$$
	for $i=1, \dots, d$ are not forbidden for $\cT^d$ by
	$$
	\sum_{k=0}^2 \exp\left(\pm \frac{2\pi\iu k}{3}\right) = 0.
	$$
	Thus $\beta$ is a magic function for $\cT^d$, showing its optimality.
\end{proof}

\section{The $5$-point Fibonacci lattice}

Let us make some preparations for the proof of Theorem \ref{thm:fib_lat_5}. We will give a construction of a suitable magic function $\beta = \beta_p$, which will be a bivariate polynomial of the form
$$
\beta_p(s_1, s_2) = \sum_{m_1, m_2 = 0}^4 \hat\beta_p(m_1, m_2) T_{m_1}(s_1) T_{m_2}(s_2).
$$
For the $5$-point Fibonacci lattice the forbidden degrees in this range are
\begin{align} \label{forbidden_degs}
	(1, 2), (1, 3), (2, 1), (2, 4), (3, 1), (3, 4), (4, 2), (4, 3).
\end{align}
Recall that for these degrees we need $\hat\beta_p(m_1, m_2) = 0$, while for general $m_1, m_2 = 0, 1, \dots, 4, (m_1, m_2) \neq (0, 0)$ we need $\hat\beta_p(m_1, m_2) \geq 0$. Also recall that we need
$$
\beta_p(s_1, s_2) \leq c_p^2(s_1, s_2)
$$
for all $-1 \leq s_1, s_2 \leq 1$ with equality in $(s_1, s_2) = (u, v), (v, u)$ where
\begin{align} \label{eq:uv}
	(u, v) \coloneqq \left(\cos\left(\frac{2\pi}5\right), \cos\left(\frac{4\pi}5\right)\right) = \left(\frac{\sqrt{5}-1}{4}, -\frac{\sqrt{5}+1}{4}\right).
\end{align}
The construction of $\beta_p$ will be done via a multi-step procedure, which we briefly summarize here:
\begin{itemize}
	\item [(1)] Construct a function $h_p(s) = \sum_{m=0}^4 \hat h_p(m) T_m(s)$ with $\hat h_p(m) \geq 0$ and
	$$
	0 \leq h_p(s) \leq f_p(s) \coloneqq 1 + \frac p2 \left(\frac16 - \frac{\arccos s}{2\pi} + \left(\frac{\arccos s}{2\pi}\right)^2\right)
	$$
	for $-1 \leq s \leq 1$, where the right inequality becomes an equality for $s = u, v$. Then
	$$
	h_p(s_1) h_p(s_2) \leq f_p(s_1) f_p(s_2) \eqcolon \gamma_p(s_1, s_2)
	$$
	with equality (in particular) for $(s_1, s_2)= (u, v), (v, u)$. Note though that the function $h_p(s_1) h_p(s_2)$ still contains forbidden degrees.
	
	\item [(2)] Construct a function $g_p(s_1, s_2) = \sum_{m_1, m_2 = 0}^4 \hat g_p(m_1, m_2) T_{m_1}(s_1) T_{m_2}(s_2)$ such that $0 \leq g_p(s_1, s_2)$ for $-1 \leq s_1, s_2 \leq 1$ with equality for $(s_1, s_2) = (u, v), (v, u)$ and
	$$
	\hat g_p(m_1, m_2) \leq \hat h_p(m_1) \hat h_p(m_2)
	$$
	with equality for the forbidden degrees $(m_1, m_2)$ as in \eqref{forbidden_degs}.
	
	\item [(3)] The function $\beta_p(s_1, s_2) \coloneqq h_p(s_1) h_p(s_2) - g_p(s_1, s_2)$ shows the optimality of $\Phi_5$ for the potential $c_p^2$.
\end{itemize}

\begin{proof} [Proof of Theorem \ref{thm:fib_lat_5}]
	Throughout the proof let $0 \leq p \leq 9$.
	
	\underline{Step 1, Hermite interpolation:} Let $h_p$ be the polynomial of degree $4$ given by the conditions
	\begin{align} \label{interpolation_conds}
		\begin{split}
			h_p(-1) & = f_p(-1) = 1-\frac p{24}, \\
			h_p(v) & = f_p(v) = 1-\frac{11p}{300}, \\
			h_p'(v) & = f_p'(v) = \frac{p}{10\pi\sqrt{\frac{5-\sqrt5}{2}}}, \\
			h_p(u) & = f_p(u) = 1+\frac p{300}, \\
			h_p'(u) & = f_p'(u) = \frac{3p}{10\pi\sqrt{\frac{5+\sqrt5}{2}}}.
		\end{split}
	\end{align}
	This is an instance of an \emph{Hermite interpolation problem}. Such a polynomial $h_p$ exists and is uniquely determined by these conditions.
	
	\underline{Step 1.1, $h_p(s) \leq f_p(s)$:} By the error formula for the Hermite interpolation \cite{Atk89, CK07}, for every $-1 \leq s \leq 1$ there is a $\xi = \xi(s) \in [-1, 1]$ such that
	\begin{align} \label{eq:hermite_error_fp}
		f_p(s) - h_p(s) = \frac{f_p^{(5)}(\xi)}{5!} (\xi+1) (\xi - v)^2 (\xi - u)^2.
	\end{align}
	Note that by the Taylor expansion \cite[Equation (13)]{Leh85} 
	$$
	f_p(s) = 1- \frac p{24} + \frac{p}{4\pi^2} \sum_{n=1}^\infty \frac{2^n}{n^2 {2n \choose n}} (s+1)^n
	$$
	for $-1 \leq s \leq 1$ we have $f_p^{(k)}(s) \geq 0$ for all $k \in \bbN_0$ in the range $0 \leq p \leq 9$ (even up to $0 \leq p \leq 24$). This property is called \emph{absolute monotonicity} of $f_p$. In particular, the right-hand side of \eqref{eq:hermite_error_fp} is nonnegative and we get 
	$h_p(s) \leq f_p(s)$ with equality for $s = -1, u, v$ by interpolation.
	
	
	\underline{Step 1.2, $\hat h_p(m) \geq 0$:} The interpolation conditions \eqref{interpolation_conds} enable one to calculate the coefficients $\hat h_p(m)$ explicitly by
	\begin{align} \label{eq:h_coeffs}
		\begin{split}
			\left[\begin{matrix}
				\hat h_p(0) \\ \hat h_p(1) \\ \hat h_p(2) \\ \hat h_p(3) \\ \hat h_p(4)
			\end{matrix}\right] & = \left[\begin{matrix}
				5 & -\frac{10+4\sqf}{5} & \frac{5+\sqf}{5} & -\frac{10-4\sqf}{5} & \frac{5-\sqf}{5} \\
				8 & -\frac{100+36\sqf}{25} & \frac{17+3\sqf}{10} & -\frac{100-36\sqf}{25} & \frac{17-3\sqf}{10} \\
				6 & -\frac{75+21\sqf}{25} & \frac{6+\sqf}5 & -\frac{75-21\sqf}{25} & \frac{6-\sqf}5 \\
				4 & -\frac{50+14\sqf}{25} & \frac{4+\sqf}5 & -\frac{50-14\sqf}{25} & \frac{4-\sqf}5 \\
				2 & -\frac{25+9\sqf}{25} & \frac{3+\sqf}{10} & -\frac{25-9\sqf}{25} & \frac{3-\sqf}{10}
			\end{matrix}\right] \left[\begin{matrix}
				f_p(-1) \\ f_p(v) \\ f_p'(v) \\ f_p(u) \\ f_p'(u)
			\end{matrix}\right] \\
			& = \left[\begin{matrix}
				1 \\
				0 \\
				0 \\
				0 \\
				0 \\
			\end{matrix}\right] + \left[\begin{matrix}
				-\frac{425-96\sqf}{3000} + \frac{(5+\sqf)\sqrt{5+\sqf} + 3 (5-\sqf)\sqrt{5 - \sqf}}{50 \pi \sqrt{10}} \\
				-\frac{1500-432\sqf}{7500} + \frac{(17+3\sqf)\sqrt{5+\sqf} + 3 (17-3\sqf)\sqrt{5 - \sqf}}{100 \pi \sqrt{10}} \\
				-\frac{1125-252\sqf}{7500} + \frac{(6+\sqf)\sqrt{5+\sqf} + 3 (6-\sqf)\sqrt{5 - \sqf}}{50 \pi \sqrt{10}} \\
				-\frac{750-168\sqf}{7500} + \frac{(4+\sqf)\sqrt{5+\sqf} + 3 (4-\sqf)\sqrt{5 - \sqf}}{50 \pi \sqrt{10}} \\
				-\frac{375-108\sqf}{7500} + \frac{(3+\sqf)\sqrt{5+\sqf} + 3 (3-\sqf)\sqrt{5 - \sqf}}{100 \pi \sqrt{10}} \\
			\end{matrix}\right] p
		\end{split}
	\end{align}
	(the columns of the matrix are the Chebyshev coefficients of the basis functions for this Hermite interpolation problem as given in \cite[Equations (3.6.2-4)]{Atk89}). These are degree $1$ polynomials in $p$ that stay nonnegative in the range $0 \leq p \leq 9$ (even up to $p \leq 315.02\dots$ after which only $\hat h_p(0)$ turns negative).
	
	\underline{Step 1.3, $0 \leq h_p(s)$:} The inequality $\hat h_p(0) \geq \hat h_p(1) + \dots + \hat h_p(4)$ is equivalent to
	\begin{align*}
		& 1-\left(\frac{425-96\sqf}{3000} - \frac{(5+\sqf)\sqrt{5+\sqf} + 3 (5-\sqf)\sqrt{5 - \sqf}}{50 \pi \sqrt{10}}\right)p \\
		\geq & \left(-\frac{3750-960\sqf}{7500} + \frac{(40+8\sqf)\sqrt{5+\sqf} + 3 (40-8\sqf)\sqrt{5 - \sqf}}{100 \pi \sqrt{10}}\right) p,
	\end{align*}
	which holds in the range $0 \leq p \leq 9$ (even up to $0 \leq p \leq 17.49\dots$) so that we can conclude (since $|T_m(s)| \leq 1$ for $-1 \leq s \leq 1$)
	$$
	h_p(s) = \sum_{m=0}^4 \hat h_p(m) T_m(s) \geq \hat h_p(0) - \sum_{k=1}^4 \hat h_p(m) \geq 0.
	$$
	
	\underline{Step 2, remove forbidden degrees:} By what we have shown so far we have
	\begin{align*} 
		\gamma_p(s_1, s_2) & \coloneqq f_p(s_1) f_p(s_2) \\
		& \geq h_p(s_1) h_p(s_2) = \sum_{m_1, m_2 = 0}^4 \hat h_p(m_1) \hat h_p(m_2) T_{m_1}(s_1) T_{m_2}(s_2).
	\end{align*}
	To obtain a suitable magic function $\beta_p$ we need to remove the forbidden degrees \eqref{forbidden_degs}. This will be done by subtracting off a suitable auxiliary function $g_p$ that kills the forbidden degrees while preserving the interpolation conditions. We choose
	\begin{align} \label{eq:kill_off}
		\begin{split}
			g_p(s_1, s_2) \coloneqq & (1+2s_1+2s_2)^2 \big(\lambda_0 + \lambda_1 (2+s_1+s_2) + \lambda_2 (2s_1^2+2s_2^2) \\
			& ~~~~~~~~~~~~~~~~~~~~~ + \lambda_3 (1-s_1 s_2) + \lambda_4 (1+2s_1^2)(1+2s_2^2)\big) \\
			& + (1+4s_1 s_2)^2 \big(\mu_0 + \mu_1 (4-2s_1^2-2s_2^2) \\
			& ~~~~~~~~~~~~~~~~~~~~ + \mu_2 (1-s_1-s_2+2s_1^2+2s_2^2+2s_1^2s_2+2s_1s_2^2)\\
			& ~~~~~~~~~~~~~~~~~~~~ + \mu_3 (2-2s_1^2-2s_2^2+4s_1^2s_2^2)\big),
		\end{split}
	\end{align}
	as an ansatz where $\lambda_0, \dots, \lambda_4, \mu_0, \dots, \mu_3 \in \bbR$ are coefficients to be chosen in a bit. Writing $g_p(s_1, s_2) = \sum_{m_1, m_2 = 0}^4 \hat g_p(m_1, m_2) T_{m_1}(s_1) T_{m_2}(s_2)$, the coefficients from \eqref{eq:kill_off} and $\hat g_p(m_1, m_2)$ are related by
	\begin{align} \label{eq:ghat_relation}
		\left[\begin{matrix}
			\hat g_p(0, 0) \\
			\hat g_p(0, 1) \\
			\hat g_p(0, 2) \\
			\hat g_p(0, 3) \\
			\hat g_p(0, 4) \\
			\hat g_p(1, 1) \\
			\hat g_p(1, 2) \\
			\hat g_p(1, 3) \\
			\hat g_p(1, 4) \\
			\hat g_p(2, 2) \\
			\hat g_p(2, 3) \\
			\hat g_p(2, 4) \\
			\hat g_p(3, 3) \\
			\hat g_p(3, 4) \\
			\hat g_p(4, 4)
		\end{matrix}\right] = \left[\begin{matrix}
			5 & 14 & 12 & 3 & 24 & 5 & 6 & 19 & 6 \\ 
			4 & 18 & 10 & 2 & 20 & 0 & 0 & 5 & 0 \\
			2 & 6 & 9 & 0 & 19 & 4 & 1 & 19 & 6 \\
			0 & 1 & 2 & 0 & 4 & 0 & 0 & 3 & 0 \\
			0 & 0 & 1 & 0 & 2 & 0 & -2 & 2 & 1\\ 
			8 & 24 & 24 & 1 & 50 & 8 & 8 & 32 & 10 \\
			0 & 6 & 4 & -2 & 10 & 0 & 0 & 9 & 0 \\
			0 & 0 & 4 & -1 & 10 & 0 & -4 & 4 & 2 \\ 
			0 & 0 & 0 & 0 & 0 & 0 & 0 & 3 & 0 \\
			0 & 0 & 4 & -2 & 13 & 4 & 0 & 20 & 9 \\
			0 & 0 & 0 & 0 & 2 & 0 & 0 & 4 & 0 \\
			0 & 0 & 0 & 0 & 1 & 0 & -2 & 2 & 2 \\ 
			0 & 0 & 0 & 0 & 2 & 0 & 0 & 0 & 2 \\
			0 & 0 & 0 & 0 & 0 & 0 & 0 & 1 & 0 \\
			0 & 0 & 0 & 0 & 0 & 0 & 0 & 0 & 1
		\end{matrix}\right] \left[\begin{matrix}
			\lambda_0 \\
			\lambda_1 \\
			\lambda_2 \\
			\lambda_3 \\
			\lambda_4 \\
			\mu_0 \\
			\mu_1 \\
			\mu_2 \\
			\mu_3
		\end{matrix}\right],
	\end{align}
	where by the symmetry relation $\hat g_p(m_1, m_2) = \hat g_p(m_2, m_1)$ all the coefficients are determined in this way. Let the coefficients in \eqref{eq:kill_off} now be given in such a way that $\hat g_p(m_1, m_2) = \hat h_p(m_1) \hat h_p(m_2)$ for the degrees
	\begin{align} \label{eq_degrees}
		(m_1, m_2) \in \{(1, 1), (1, 2), (1, 3), (2, 2), (2, 4), (3, 3), (3, 4)\}
	\end{align}
	as well as
	$$
	\hat g_p(0, 4) = \frac{p}9 \hat h_p(0) \hat h_p(4)
	$$
	and
	$$
	\hat g_p(4, 4) = \left(1-\frac{p}9\right) \hat h_p(4)^2.
	$$
	These 9 conditions uniquely determine the 9 coefficients $\lambda_0, \dots, \mu_3$. Indeed, by solving the subsystem in \eqref{eq:ghat_relation} corresponding to the rows of these degrees we obtain
	\begin{align} \label{eq:g_coeffs}
		\begin{split}
			& \left[\begin{matrix}
				\lambda_0 \\
				\lambda_1 \\
				\lambda_2 \\
				\lambda_3 \\
				\lambda_4 \\
				\mu_0 \\
				\mu_1 \\
				\mu_2 \\
				\mu_3
			\end{matrix}\right] \\
			= \frac1{24} & \left[\begin{matrix}
				-156 & 3 & -12 & 39 & -6 & 90 & -60 & 108 & 42 \\
				16 & 0 & 4 & -8 & 0 & 0 & 4 & -36 & -8 \\
				24 & 0 & 0 & 0 & 0 & -24 & -12 & 0 & 48 \\ 
				96 & 0 & 0 & -24 & 0 & -48 & 48 & 0 & -48 \\ 
				0 & 0 & 0 & 0 & 0 & 0 & 12 & 0 & -24 \\
				24 & 0 & 0 & -12 & 6 & 0 & -3 & -120 & -48 \\ 
				0 & 0 & 0 & 0 & 0 & -12 & 6 & 24 & 12 \\
				0 & 0 & 0 & 0 & 0 & 0 & 0 & 24 & 0 \\
				0 & 0 & 0 & 0 & 0 & 0 & 0 & 0 & 24 
			\end{matrix}\right] \left[\begin{matrix}
				\frac p9 \hat h_p(0) \hat h_p(4) \\
				\hat h_p(1)^2 \\
				\hat h_p(1) \hat h_p(2) \\
				\hat h_p(1) \hat h_p(3) \\
				\hat h_p(2)^2 \\
				\hat h_p(2) \hat h_p(4) \\
				\hat h_p(3)^2 \\
				\hat h_p(3) \hat h_p(4) \\
				\left(1-\frac{p}9\right) \hat h_p(4)^2
			\end{matrix}\right].
		\end{split}
	\end{align}
	
	\underline{Step 2.1, $g_p(s_1, s_2) \geq 0$ for $-1 \leq s_1, s_2 \leq 1$:} It suffices to show that the coefficients $\lambda_0, \dots, \mu_3$ are nonnegative. Indeed, the functions after the coefficients $\lambda_1, \lambda_2, \lambda_3, \lambda_4, \mu_1$ are clearly nonnegative on $[-1, 1]^2$. For $\mu_2$ and $\mu_3$ observe that
	\begin{align*}
		& 1-s_1-s_2+2s_1^2+2s_2^2+2s_1^2s_2+2s_1s_2^2 \\
		= & \frac5{18}+\frac1{72}(2-s_1-s_2)+\frac23 (1-s_1)(1-s_2)\\
		& +2(1+s_1)\left(s_2-\frac1{12}\right)^2 + 2(1+s_2) \left(s_1-\frac1{12}\right)^2 \geq 0
	\end{align*}
	and
	\begin{align*}
		2 - 2s_1^2 - 2s_2^2 + 4 s_1^2 s_2^2 = 2(1-s_1^2)(1-s_2^2)+2s_1^2s_2^2 \geq 0
	\end{align*}
	for $-1 \leq s_1, s_2 \leq 1$. All coefficients $\lambda_0, \dots, \mu_3$ are polynomials of the form $(c_0 + c_1 p) p^2$ in $p$ (with coefficients $c_0, c_1 \in \bbR$ depending on which coefficient is examined). These coefficients can be determined in closed form via \eqref{eq:h_coeffs} and \eqref{eq:g_coeffs}, here we give them in decimal form up to $5$ significant digits:
	\begin{align*}
		\lambda_0 & = p^2 (0.000029101\ldots + 0.00000047800\dots p), \\
		\lambda_1 & = p^2 (0.000048272\ldots - 0.000000048255\dots p), \\
		\lambda_2 & = p^2 (0.000020851\ldots - 0.000000084905\dots p), \\
		\lambda_3 & = p^2 (0.000025894\ldots - 0.00000028953\dots p), \\
		\lambda_4 & = p^2 (0.0000011872\ldots + 0.0000000050091\dots p), \\
		\mu_0 & = p^2 (0.00000062987\ldots - 0.000000064869\dots p), \\
		\mu_1 & = p^2 (0.00000017306\ldots - 0.0000000025045\dots p), \\
		\mu_2 & = 0.00000033334\dots p^2, \\
		\mu_3 & = 0.0000000050091\dots p^2 (9-p).
	\end{align*}
	Even up to this precision it is clear that these expressions are nonnegative for $0 \leq p \leq 9$.
	
	\underline{Step 2.2, $g_p(u, v) = g_p(v, u) = 0$:} This follows directly from the relations $1+2u+2v = 1+4uv = 0$ as can be seen from \eqref{eq:uv}.
	
	\underline{Step 2.3, $\hat g_p(m_1, m_2) \leq \hat h_p(m_1) \hat h_p(m_2)$:} By construction it holds $\hat g_p(m_1, m_2)$ $= \hat h_p(m_1) \hat h_p(m_2)$ for the degrees as in \eqref{eq_degrees}, including the forbidden degrees \eqref{forbidden_degs}. Again by construction it also holds $\hat g_p(m_1, m_2) \leq \hat h_p(m_1) \hat h_p(m_2)$ for $(m_1, m_2) \in \{(0, 4), (4, 4)\}$. For the remaining degrees $(0, 1), (0, 2), (0, 3), (1, 4),$ $(2, 3)$ (there is no restriction on the degree $(0, 0)$) we have, combining \eqref{eq:ghat_relation} and \eqref{eq:g_coeffs},
	\begin{align*}
		& \left[\begin{matrix}
			\hat g_p(0, 1) \\
			\hat g_p(0, 2) \\
			\hat g_p(0, 3) \\
			\hat g_p(1, 4) \\
			\hat g_p(2, 3)
		\end{matrix}\right] \\
		= \frac1{24} & \left[\begin{matrix}
			96 & 12 & 24 & -36 & -24 & 24 & 48 & -96 & -72 \\
			96 & 6 & 0 & -18 & 12 & -48 & 18 & 0 & -24 \\
			64 & 0 & 4 & -8 & 0 & -48 & 28 & 36 & -8 \\
			0 & 0 & 0 & 0 & 0 & 0 & 0 & 72 & 0 \\
			0 & 0 & 0 & 0 & 0 & 0 & 24 & 96 & -48 
		\end{matrix}\right] \left[\begin{matrix}
			\frac p9 \hat h_p(0) \hat h_p(4) \\
			\hat h_p(1)^2 \\
			\hat h_p(1) \hat h_p(2) \\
			\hat h_p(1) \hat h_p(3) \\
			\hat h_p(2)^2 \\
			\hat h_p(2) \hat h_p(4) \\
			\hat h_p(3)^2 \\
			\hat h_p(3) \hat h_p(4) \\
			\left(1-\frac{p}9\right) \hat h_p(4)^2
		\end{matrix}\right].
	\end{align*}
	The expressions $\hat h_p(m_1) \hat h_p(m_2) - \hat g_p(m_1, m_2)$ for these degrees are polynomials of the form $(c_0+c_1p+c_2p^2)p$ in $p$ (with coefficients $c_0, c_1, c_2 \in \bbR$ depending on the degree $(m_1, m_2)$). These expressions can be determined similar to step 2.1, we again give them up to $5$ significant digits:
	\begin{align*}
		\hat h_p(0) \hat h_p(1) - \hat g_p(0, 1) & = p (0.044660\ldots - 0.0014128\dots p \\
		& ~~~~~~ + 0.00000028452\dots p^2), \\
		\hat h_p(0) \hat h_p(2) - \hat g_p(0, 2) & = p (0.0075261\ldots - 0.00059124\dots p \\
		& ~~~~~~ + 0.00000029454\dots p^2), \\
		\hat h_p(0) \hat h_p(3) - \hat g_p(0, 3) & = p (0.0015699\ldots - 0.00010070\dots p \\
		& ~~~~~~ + 0.00000019803\dots p^2), \\
		\hat h_p(1) \hat h_p(4) - \hat g_p(1, 4) & = 0.0000084826\dots p^2, \\
		\hat h_p(1) \hat h_p(4) - \hat g_p(1, 4) & = p^2 (0.0000081077\ldots - 0.000000010018\dots p).
	\end{align*}
	It is again straight forward to check that these expressions are nonnegative for $0 \leq p \leq 9$.
	
	\underline{Step 3, constructing $\beta_p$:} Let now
	$$
	\beta_p(s_1, s_2) \coloneqq h_p(s_1) h_p(s_2) - g_p(s_1, s_2).
	$$
	We show that this is a magic function as in Theorem \ref{thm:magic_function} for the potential $c_p^2$ and the $5$-point Fibonacci lattice $\Phi_5$. Indeed, by steps 1.1, 1.3 and 2.1 we have $\gamma_p \geq \beta_p$ with equality in $(s_1, s_2) = (u, v), (v, u)$ by steps 1.1 and 2.2. The Chebyshev coefficients of $\beta_p$ are nonnegative (we have no constraint on $\hat\beta_p(0, 0)$) by step 2.3 and even zero for the forbidden degrees \eqref{forbidden_degs}. Thus $\beta_p$ is a magic function as in Theorem \ref{thm:magic_function} for $\gamma_p$.
\end{proof}

\begin{rem}
	The function $\gamma_p(s_1, s_2)$ is convex for $0 \leq p \leq \frac{24\pi^2}{18+\pi^2} = 8.4992\dots$ (by logarithmic convexity of $f_p$ \cite{SS12}), so that the range given for $p$ in Theorem \ref{thm:fib_lat_5} also captures the convex case similar to Theorem \ref{thm:3_pt_lat}.
\end{rem}


\section{Summary and outlook}

Motivated by discrepancy theory and related topics we introduced tensor product energies on the torus. We used linear programming methods to show the optimality of certain point sets. To this end we showed that the rational $3$-point lattice in any dimension fulfills a universal optimality statement with respect to monotone and convex potentials (Theorem \ref{thm:3_pt_lat}). In the special case of the two-dimensional torus we showed the optimality of the $5$-point Fibonacci lattice for a continuously parametrized set of potentials related to the worst case error of the quasi-Monte Carlo method over certain Sobolev spaces of smoothness $1$ (Theorem \ref{thm:fib_lat_5}).

There are several open question that remain concerning this topic. The main one is whether general Fibonacci lattices also fulfill a universal optimality condition and, if so, if this can also be shown for special point sets in the $d$-dimensional torus. Another question would be to generalize these results to other spaces. Notably, interpreting the $d$-dimensional torus as the $d$-fold product of one-dimensional spheres $S^1$, a natural next step would be to consider general product spaces of the form $S^{n_1} \times \dots \times S^{n_d}$. Further extensions to other non-two-point homogeneous spaces such as Grassmannians \cite{CHS96, EL25} or discrete spaces \cite{Bar21, BS21} are also imaginable.

\section*{Acknowledgement}

The author is supported by the ESF Plus Young Researchers Group ReSIDA-
H2 (SAB, project number 100649391). The author would also like to thank Dmitriy Bilyk from the University of Minnesota for many interesting and helpful conversations ultimately leading to this manuscript as well as Melia Haase from the University of Technology Chemnitz for assistance concerning numerical calculations.

\end{document}